\documentclass[12pt]{article}
\usepackage{amsmath,amsthm,amssymb,epsfig}

\usepackage[margin=25mm]{geometry}

\usepackage{array}
\newcolumntype{C}{>{\centering\arraybackslash}p{12pt}}

\newtheorem{theorem}{Theorem}
\newtheorem{lemma}{Lemma}
\newtheorem{corollary}[theorem]{Corollary}
\newtheorem{conj}{Conjecture}

\newtheoremstyle{example}{\topsep}{\topsep}%
     {}%         Body font
     {}%         Indent amount (empty = no indent, \parindent = para indent)
     {\bfseries}% Thm head font
     {}%        Punctuation after thm head
     {\newline}%     Space after thm head (\newline = linebreak)
     {\thmname{#1}\thmnumber{ #2}\thmnote{ #3}}%         Thm head spec
\theoremstyle{example}

\def\clap#1{\hbox to 0pt{\hss #1\hss}}

\def\dfrac#1#2{\lower0.15ex\hbox{\large$\frac{#1}{#2}$}} % bdm's tricky fracs
\renewcommand{\geq}{\geqslant}
\renewcommand{\leq}{\leqslant}
\renewcommand{\ge}{\geqslant}
\renewcommand{\le}{\leqslant}

\newlength{\mycolwd}% array column width
\def\mtee{\makebox[\mycolwd]{$\cdot$}}

\def\iwb{\rightarrow\infty}
\def\sym{\mathcal{S}}

\def\eref#1{$(\ref{#1})$}

\def\sref#1{\S$\ref{#1}$}
\def\lref#1{Lemma~$\ref{#1}$}
\def\tref#1{Theorem~$\ref{#1}$}
\def\cref#1{Corollary~$\ref{#1}$}
\def\cjref#1{Conjecture~$\ref{#1}$}

\def\<{\langle}
\def\>{\rangle}

\title{Latin squares with no transversals}

\author{Nicholas J.\ Cavenagh\\
\small Department of Mathematics\\[-0.75ex]
\small University of Waikato\\[-0.75ex]
\small Private Bag 3105\\[-0.75ex]
\small Hamilton, New Zealand\\
\small\tt nickc@waikato.edu.au
\and 
Ian M.\ Wanless\thanks{Research supported by ARC grant DP150100506.}\\
\small School of Mathematical Sciences\\[-0.75ex]
\small Monash University\\[-0.75ex]
\small Vic 3800, Australia\\
\small\tt ian.wanless@monash.edu
}

\date{}
\begin{document}

\maketitle

\begin{abstract}
A $k$-plex in a latin square of order $n$ is a selection of $kn$ 
entries that includes $k$ representatives from each row and column
and $k$ occurrences of each symbol. A $1$-plex is also known as a
transversal.

It is well known that if $n$ is even then $B_n$, the addition table
for the integers modulo $n$, possesses no transversals. We show that
there are a great many latin squares that are similar to $B_n$ and have no
transversal. As a consequence, the number of species of
transversal-free latin squares is shown to be at least
$n^{n^{3/2}(1/2-o(1))}$ for even $n\iwb$.

We also produce various constructions for latin squares that have no
transversal but do have a $k$-plex for some odd $k>1$. We prove a 2002
conjecture of the second author that for all even orders $n>4$ there is a
latin square of order $n$ that contains a $3$-plex but no transversal.
We also show that for odd $k$ and $m\geq 2$, there exists a latin
square of order $2km$ with a $k$-plex but no $k'$-plex for odd $k'<k$.

\end{abstract}

\section{Introduction}\label{s:intro} 

A $k\times n$ latin rectangle is an array containing $n$ symbols such
that each symbol occurs once in each row and at most once in each
column.  A latin square of order $n$ is an $n\times n$ latin
rectangle.  We index rows and columns by the set
$N_n=\{0,1,\dots,n-1\}$ and also use $N_n$ for our symbols.  A latin
square may then be specified as a set of ordered (row, column, symbol)
triples called {\em entries}.  Each entry is an element of $N_n\times
N_n\times N_n$.  This viewpoint allows a natural action of the wreath
product $\sym_n\wr\sym_3$ on the latin squares of order $n$, where
$\sym_n$ denotes the symmetric group of degree $n$.  Orbits under this
action are known as {\em species} (sometimes also called 
{\em main classes}).

A $k$-plex of a latin square of order $n$ is a selection of $kn$
entries such that exactly $k$ entries are chosen from each row and
each column, and each symbol is chosen $k$ times.  A $1$-plex is also
known as a {\em transversal} and a $3$-plex is known as a triplex. A
$k$-plex is said to be an {\em odd plex} if $k$ is odd. 
See \cite{transurv} for a survey on
transversals and plexes more generally.  

One of our goals is to show that there are very many transversal-free
Latin squares for each even order. We do this in the next section.
Our second major goal is to prove the following conjecture from
\cite{origplex}, which we do in \sref{s:triplexnotrans}.

\begin{conj}\label{cj:triplexnotrans}
For all even $n>4$ there is a latin square of order $n$ that contains
a triplex but no transversal.
\end{conj}

It is clear that this conjecture cannot extend to $n=4$, since the
complement of a transversal is always an $(n-1)$-plex. Part of the
motivation for studying plexes comes from their use in creating
orthogonal partitions in experiment design.  Basic existence questions
for plexes seem to be very difficult. It has been conjectured in
\cite{origplex} that every latin square of order $n$ has a $k$-plex
for every even $k\le n$. There seems to be more diversity regarding
existence of odd plexes, and our result adds to the
possibilities that are known to occur.

The following lemma will be crucial to our work. It is one variant of
a result known as the Delta lemma which has been employed in several
papers including \cite{EW08,EW11}. See \cite{transurv} for a
discussion of other applications. 

\begin{lemma}\label{l:Delta}
Let $L$ be a latin square of even order $n$ indexed by $N_n$. 
Suppose that $m$ is an odd divisor of $n$.
We define a function $\Delta_m$ on the entries of $L$ by
specifying that $\Delta_m(r,c,s)$ is the 
integer of least absolute value that satisfies
\begin{equation*}\label{e:deltadef}
\Delta_m(r,c,s)
\equiv \left\lfloor\frac{s}{m}\right\rfloor
-\left\lfloor\frac{r}{m}\right\rfloor
-\left\lfloor\frac{c}{m}\right\rfloor\mod \frac nm.
\end{equation*}
Let $k$ be an odd positive integer.
If $K$ is a $k$-plex of $L$ then, 
\begin{equation*}
\sum_{(r,c,s)\in K}\Delta_m(r,c,s)\equiv
\frac{n}{2m}\mod\frac nm.
\end{equation*}
\end{lemma}

\begin{proof} By definition,
\begin{align*}
\sum_{(r,c,s)\in K}\Delta_m(r,c,s)
&\equiv
k\sum_{s=0}^{n-1}\left\lfloor\frac{s}{m}\right\rfloor
-k\sum_{r=0}^{n-1}\left\lfloor\frac{r}{m}\right\rfloor
-k\sum_{c=0}^{n-1}\left\lfloor\frac{c}{m}\right\rfloor\\
&\equiv
-km\sum_{i=0}^{n/m-1}i
=-km\left(\frac{n}m-1\right)\frac{n}{2m}\\
&\equiv\frac{n}{2m}\mod\frac nm,
\end{align*}
since $km$ is odd and $n/m$ is even.
\end{proof}

Most applications simply use $m=1$, but we will need the more general
version in the next section. Also, the requirement that $\Delta_m$
has the least absolute value in its residue class will be vital in 
inequalities throughout the paper.

The number of transversals is invariant within a species.
We call a latin square {\it transversal-free} if it possesses no transversals.
Define $B_n$ to be the addition table for the integers modulo $n$.
It is immediate from \lref{l:Delta} (using $m=1$)
that $B_n$ is transversal-free when
$n$ is even. In turns out that many latin squares that resemble $B_n$
are also transversal-free. We will use two results on this theme.
One is a new result that we prove in the next section.
The other is the following classical result due to Maillet~\cite{Mai94},
which was generalised from transversals to all odd plexes in \cite{origplex}.

%(c.f.~\cite[Thm~3.1]{transurv}):

\begin{lemma}\label{l:steptype}
Let $n=bm$ where $b$ is even and $m$ is odd.
Let $L=[L_{ij}]$ be a latin square indexed by $N_n$. 
%If there exists a surjection $\phi:N_n\rightarrow\Z_{b}$ such that
%$\phi(L_{ij})\equiv\phi(i)+\phi(j)\mod b$ 
If 
\begin{equation}\label{e:steptype}
\lfloor L_{ij}/m\rfloor\equiv\lfloor i/m\rfloor+\lfloor j/m\rfloor\mod b
\end{equation} 
for all $i,j\in N_n$, then $L$ has no odd plexes.
\end{lemma}

\lref{l:steptype} is an
immediate consequence of \lref{l:Delta}. 
Latin squares that satisfy \eref{e:steptype} for all $i,j\in N_n$ are
sometimes said to be of {\em``step-type''}. 
%It was noted in \cite{origplex} that \lref{l:steptype} can be
%strengthened to conclude that $L$ has no odd plexes.
Part of the original
motivation for \cjref{cj:triplexnotrans} was to find a family of latin
squares that are transversal-free but are structurally different to
step-type latin squares. That goal was achieved in \cite{EW08}, but
without proving the original conjecture. It was shown that there are
latin squares that have $k$-plexes for some odd $k$ but not for any small 
odd $k$. By proving \cjref{cj:triplexnotrans}, we demonstrate yet
another possible structure.

The other background result that we need is the following,
from \cite[p.186]{vLW01}:

\begin{lemma}\label{l:extLR}
Let $R$ be a $k\times n$ latin rectangle. The number of $n\times n$
latin squares obtained by adding rows to $R$ is at least 
\[
\prod_{i=k}^{n-1}n!(1-i/n)^n=n!^{\,n-k}(n-k)!^{\,n}/n^{n(n-k)}.
\]
\end{lemma}

\section{Number of latin squares with no transversals}\label{s:numnotrans}

It was famously conjectured by Ryser (see \cite{transurv} for 
the history of this conjecture) that all latin squares
of odd order have transversals. Our aim in this section is to show
that there are a great many different species of latin squares of even
order that have no transversals.

\begin{theorem}\label{t:manystepstype}
Let $n=2^am$ for positive integers $a$ and $m$, where $m$ is odd.
There are at least $(m/e^2)^{n^2}$ latin squares of order $n$ that
have no odd plexes.
\end{theorem}

\begin{proof}
By \lref{l:steptype}, we can construct an order $n$ latin square with
no odd plexes by patching together $2^{2a}$ latin subsquares of order
$m$. For each subsquare we have at least $m!^{\,2m}/m^{m^2}$ choices,
by \lref{l:extLR}, and these choices can be made independently. The
result now follows from Stirling's approximation, given that
$m!>(m/e)^m$ for all $m\ge1$.
\end{proof}

\begin{corollary}
For $m\iwb$ with fixed $a\ge1$, the number of species of transversal-free
latin squares of order $n=2^am$ is at least $n^{n^2(1-o(1))}$.
\end{corollary}

\begin{proof}
The number of transversal-free latin squares is at least
$(cn)^{n^2}$ for the constant $c=2^{-a}e^{-2}>0$.
The result now follows, since the number of
latin squares in each species is at most $6(n!)^3=n^{O(n)}$.
\end{proof}

Of course, \tref{t:manystepstype} does not tell us much for orders
that are powers of 2. This is unavoidable because \lref{l:steptype} only
applies to one species of such an order, namely the species 
containing $B_n$.  Our next
result will allow us to show that there are many species of
transversal-free latin squares for {\em all} even orders.

\begin{theorem}\label{t:botrows}
Let $n=bm$ where $b$ is even and $m$ is odd. Let $r$ be a nonnegative
integer and $k$ an odd positive integer. 
Suppose $L$ is any latin square of order $n$ indexed by $N_n$,
which satisfies \eref{e:steptype} for all $j\in N_n$ and $0\le i<n-mr$.
If $km^2r(r-1)<n$ then $L$ has no $k$-plexes.
\end{theorem}

\begin{proof}
Suppose that $K$ is a $k$-plex of $L$.
For $0\le i<n$, let $\{e_{i,a}:0\le a<k\}$ be the set of 
entries in $K$ from the $i$-th row of $L$.
By assumption, $\Delta_m(e_{i,a})=0$ whenever $0\le i<n-mr$. 
Now consider some $e_{i,a}=(i,c,s)$ where $i\ge n-mr$. 
%The symbol $s_i$ must match a symbol that occurs in 
%the last $r$ rows of column $c_i$ of $B_n$. 
Given that $s$ cannot match any of the symbols which are used in the 
first $n-mr$ rows of column $c$, we know that
%\[  %% Actually no; this is only true up to sequence mod b, so messy to write
%b-r+\lfloor c_i/m\rfloor\le\lfloor s_i/m\rfloor\le b-1+\lfloor c_i/m\rfloor
%\]
%and hence
$b-r-\lfloor i/m\rfloor\le \Delta_m(e_{i,a}) \le b-1-\lfloor i/m\rfloor$.
It follows that 
\[
\left|\sum_{i=0}^{n-1}\sum_{a=0}^{k-1}\Delta_m(e_{i,a})\right|
\le\sum_{a=0}^{k-1}\left|\sum_{i=n-mr}^{n-1}\Delta_m(e_{i,a})\right|
\le km\sum_{j=0}^{r-1}j=kmr(r-1)/2.
\]
The result now follows from \lref{l:Delta}.
\end{proof}

In particular, when a step-type latin square of order $n$ is
transversal-free, this property is quite robust in the sense that no
transversal will be introduced by arbitrarily changing any (roughly)
$\sqrt{n}$ consecutive rows. Putting $k=m=1$ in \tref{t:botrows}, we
see that:

\begin{corollary}
For even $n$ there are no transversals in 
any latin square which agrees with $B_n$ outside of some set
of $\lfloor\sqrt{n}\rfloor$ consecutive rows.
\end{corollary}

\begin{corollary}
For even $n\iwb$, there are at least
$n^{n^{3/2}(1/2-o(1))}$ species of transversal-free latin squares of order $n$.
\end{corollary}

\begin{proof}
Let $s=\lfloor\sqrt{n}\rfloor$.  By \lref{l:extLR} and Stirling's
approximation, there are at least
$(s!)^n(n!)^s/n^{sn}=n^{n^{3/2}(1/2-o(1))}$
ways to complete the first $n-s$ rows of
$B_n$ to a latin square.  Again, when we divide by
$n^{O(n)}$, the maximum number of latin squares in a species, this
factor gets absorbed in the error term.
\end{proof}

For $m=1$ and $r\le2$, \tref{t:botrows} is not useful since in this case
the only way to change $r$ consecutive rows is to permute them,
which does not change the species. Hence the smallest case when
\tref{t:botrows} is interesting is when $n=8$ and $r=3$. 
There are 264 latin squares that agree with $B_8$ in the first $5$
rows, and these fall in $9$ distinct transversal-free
species. Representatives of the 
$8$ species other than $B_8$ may be defined by specifying the 
non-zero values of the $\Delta_1$ function in the last three rows:
\begin{align*}
&
\left(\begin{array}{cccccccc}
\mtee&  2  &\mtee&  2  &\mtee&  2  &\mtee&  2  \\
\mtee&\mtee&\mtee&\mtee&\mtee&\mtee&\mtee&\mtee\\
\mtee& -2  &\mtee& -2  &\mtee& -2  &\mtee& -2  \\
\end{array}\right)
\left(\begin{array}{cccccccc}
\mtee&  2  &\mtee&  2  &\mtee&  2  &\mtee&  2  \\
\mtee&\mtee&\mtee&\mtee&\mtee&\mtee&  1  & -1  \\
\mtee& -2  &\mtee& -2  &\mtee& -2  & -1  & -1  \\
\end{array}\right)
\\
&
\left(\begin{array}{cccccccc}
\mtee&  2  &\mtee&  2  &\mtee&  2  &\mtee&  2  \\
\mtee&\mtee&  1  & -1  &\mtee&\mtee&  1  & -1  \\
\mtee& -2  & -1  & -1  &\mtee& -2  & -1  & -1  \\
\end{array}\right)
\left(\begin{array}{cccccccc}
\mtee&  2  &\mtee&  1  &  2  &\mtee&  1  &  2  \\
\mtee&\mtee&\mtee&  1  & -1  &\mtee&  1  & -1  \\
\mtee& -2  &\mtee& -2  & -1  &\mtee& -2  & -1  \\
\end{array}\right)
\\
&\left(\begin{array}{cccccccc}
\mtee&  2  &\mtee&  1  &  2  &\mtee&  1  &  2  \\
\mtee&\mtee&\mtee&  1  & -1  &  1  & -1  &\mtee\\
\mtee& -2  &\mtee& -2  & -1  & -1  &\mtee& -2  \\
\end{array}\right)
\left(\begin{array}{cccccccc}
\mtee&  2  &\mtee&  1  &  2  &\mtee&  1  &  2  \\
\mtee&\mtee&  1  & -1  &\mtee&\mtee&  1  & -1  \\
\mtee& -2  & -1  &\mtee& -2  &\mtee& -2  & -1  \\
\end{array}\right)
\\
&
\left(\begin{array}{cccccccc}
\mtee&  2  &\mtee&  1  &  1  &  1  &  1  &  2  \\
\mtee&\mtee&\mtee&  1  & -1  &  1  & -1  &\mtee\\
\mtee& -2  &\mtee& -2  &\mtee& -2  &\mtee& -2  \\
\end{array}\right)
\left(\begin{array}{cccccccc}
\mtee&  2  &\mtee&  1  &  1  &  2  &\mtee&  2  \\
\mtee&\mtee&\mtee&  1  & -1  &\mtee&  1  & -1  \\
\mtee& -2  &\mtee& -2  &\mtee& -2  & -1  & -1  \\
\end{array}\right)
\end{align*}

It is clear from \tref{t:botrows} that there are a great many species
of transversal-free latin squares of even order. The true number is
still far from known, but we would expect it to be negligible compared
to the number of all latin squares.

\section{Latin squares with an odd plex but no transversal}

The remainder of the paper is devoted to constructions of latin
squares that have no transversal but do have $k$-plexes for at least
one odd value of $k$. Our strategy will always be to start with $B_n$
(which has no odd plexes by \lref{l:Delta}). In $B_n$ we will locate a
structure $J$ which is close to being a $k$-plex. 
Then we will alter $B_n$ slightly and in the process
relocate a few entries from $J$ in order to make it into a $k$-plex.
%Then we will find a trade (NOT YET DEFD) in $B_n$ which moves a few entries 
%from $J$ in order to make it into a $k$-plex.
Since we always begin with $B_n$ it is convenient to define a notation
$(x;y)$ to be the triple $(x,y,z)\in N_n\times N_n\times N_n$ for which 
$z\equiv x+y\mod n$. We also adopt the convention that the result
of all calculations for indices will be reduced mod $n$ to an element
of $N_n$.

Our method for changing $B_n$ will be to use the well known theory
of latin trades (see, for example, the survey \cite{tradsurv}).
A {\em latin trade} in $B_n$ is a subset $Q$ of the entries of $B_n$ that
can be removed and replaced by a disjoint set $Q'$ of entries to
produce a new latin square. The set $Q'$ is known as the 
{\em disjoint mate} for $Q$. The sets $Q,Q'$ are sets of entries
(triples) with the property that $\pi(Q)=\pi(Q')$ for any
of the three projections $\pi$ onto two coordinates.
Checking that our latin trades and their mates have this property will be
left as a routine exercise to the reader.

\bigbreak

For the remainder of this section, $k$ is odd and $n=2km$ for some
integer $m\ge2$.  Our aim is to establish the existence of latin
squares of order $n$ which contain a $k$-plex but no smaller odd
plexes.

\bigbreak

We start by identifying a set of entries inside $B_n$, 
which we denote by $J$. 
%Since $J\subset B_n$ it suffices to identify only the cells of $J$ which are filled. 
We let $J=J_0\cup J_1\cup J_2\cup J_3$ where:
\begin{align*}
J_0 & = \big\{(i;2jm+i-1):  1\leq i\leq m,\; 0\leq j\leq k-1 \big\} \\
% WAS J_1 & = \big\{(i;(2j+1)m+i):  m\leq i\leq 2m-1,\; 0\leq j\leq k-1 \big\} \\
J_1 & = \big\{(i;2jm+i):  m\leq i\leq 2m-1,\; 0\leq j\leq k-1 \big\} \\
J_2 & = \big\{(2m(2\ell-1)+i;2jm+i): 0 \leq i\leq 2m-1,\; 0\leq j\leq k-1,\; 1\leq \ell\leq (k-1)/2  \big\}  \\
J_3 & = \big\{(4m\ell+i;2jm+i+1): 0 \leq i\leq 2m-1,\; 0\leq j\leq k-1,\; 1\leq \ell\leq (k-1)/2  \big\}.
\end{align*}

Below we exhibit $J$ when $k=m=3$. 
\[
\begin{array}{|C|C|C|C|C|C|C|C|C|C|C|C|C|C|C|C|C|C|}
\hline
&  &  &  &  & & & & & & & & & & & & & \\
\hline
1&  &  &  &  & & 7 & & & & & & 13 & & & & & \\
\hline
& 3 &  &  &  &  &  & 9 & & & & & & 15 & & & & \\
\hline
&  &  5 & 6 &  &  &  &  & 11 & 12 & &  &  &  & 17  & 0 & & \\
\hline
&  &  &  & 8 &  &  & & & & 14 & & & & &  & 2 & \\
\hline
 &  &  &  &  & 10 &  & & & & & 16 & & &  &  &  &  4\\
\hline
6 &  &  &  &  &  & 12  & & & & & & 0 & &  &  & &  \\
\hline
 &  8 &  &  &   &  &  & 14 & & & & & & 2 &  &  & &  \\
\hline
&  &  10 &  &    &     &      &    &  16 &   &   & & & & 4 &  & &  \\
\hline
 &  &  & 12  &    &     &      &   &   & 0 &    & & & &  &  6 &  &  \\
\hline
 &  &  &  & 14   &     &      &  &    &  & 2  &   & & &  &  & 8 &   \\
\hline
 &  &  &  &    &   16  &      &      &  & &  & 4  &  & &   &  & & 10 \\
\hline
 & 13 & &  &    &     &      &  1    &  &   & & &  & 7 &  &  & &  \\
\hline
 &  & 15 &  &    &     &      &     & 3  &   &  &  & &  & 9 &   & &  \\
\hline
 &  & &  17 &    &    &      &      &  &  5 &  &  &  & &   & 11   & &  \\
\hline
 &  & &  &     1   &     &      &  &   &  & 7 &  &  & &   & & 13 &   \\
\hline
 &  & &  &       &  3  &      &  &   &  &  & 9 &  &   & & & & 15   \\
\hline
17 &   &  &    &    &    & 5   &   &   &  &  & & 11 &  &   &  & &  \\
\hline
\end{array}
\]

\begin{lemma}\label{l:almost2}
Each column of $B_n$ contains precisely $k$ elements of $J$. Each
symbol in $N_n$ appears in precisely $k$ elements of $J$. Each row of $B_n$
contains precisely $k$ elements of $J$, except for the first row which
contains no elements and row $m$ which contains $2k$ elements.
\end{lemma}

\begin{proof}
The claim about rows is straightforward to check.  Each column appears
once in $J_0\cup J_1$, $(k-1)/2$ times in $J_2$ (once for each choice
of $\ell$) and $(k-1)/2$ times in $J_3$ (again, once for each choice
of $\ell$).  Each even symbol occurs once in $J_1$ and $k-1$
times in $J_2$ (twice for each choice of $\ell$).  Each odd
symbol occurs once in $J_0$ and $k-1$ times in $J_3$
(twice for each choice of $\ell$).
\end{proof}

Our aim is to find a latin trade in $B_n$ which allows us to shift precisely
$k$ elements of $J$ from row $m$ to row $0$ without making further
changes to $J$. This will allow us to show:

\begin{theorem}\label{t:kk2}
Let $k$ be odd and $m\geq 2$. Then there exists a latin square of order $n=2km$ 
with a $k$-plex but no $k'$-plex for odd $k'<k$.
\end{theorem}

\begin{proof}
Observe that $\big\{(0;jm),(m;jm): 0\leq j< 2k\big\}$ defines
a latin trade $T$ in $B_n$. The disjoint mate $T'$ of $T$ is formed by
swapping the two symbols in each column of $T$.  Moreover, 
$T\cap J=\big\{ (m;m+2mj): 0\leq j\leq k-1 \big\}$ since $m\geq 2$. 

Thus replacing $T$ with $T'$ in $B_n$ has the effect of shifting $k$
entries of $J$ from row $m$ to row $0$. From \lref{l:almost2},
$L':=(B_n\setminus T)\cup T'$ contains a $k$-plex.

Suppose that $L'$ contains a $k'$-plex $K$ for some odd $k'$ such that
$k'<k$.  For each $(r,c,s)\in T'$, $\Delta_1(r,c,s)=m$ if $r=0$ and
$\Delta_1(r,c,s)=-m$ if $r=m$.  Thus
$$\left\lvert\sum_{\smash{(}r,c,s\smash{)\in K}} \Delta_1(r,c,s)\right\lvert 
\leq mk'<n/2.$$
Hence by \lref{l:Delta}, there is no $k'$-plex in $L'$.  
\end{proof}

In the extreme case $m=2$, the previous theorem implies the existence
of a latin square of order $4k$ with a $k$-plex but no smaller odd
plexes; such a structure was first shown to exist in \cite{EW08}.
Note that the $k$-plex constructed in \tref{t:kk2} is necessarily
indivisible in the sense that it cannot be partitioned into two or more
smaller plexes. In \cite{EW11}, it was shown that for all $n\notin{2,6}$,
if $k$ is any proper divisor of $n$ then there exists a latin square of 
order $n$ that can be partitioned into indivisible $k$-plexes. That 
result is in the spirit of \tref{t:kk2}, though neither implies the other.

\section{Latin squares with a triplex but no transversal}\label{s:triplexnotrans}

In this section we prove \cjref{cj:triplexnotrans}.
%show that for any even $n>4$, there exists a latin
%square of order $n$ with a triplex but no transversal.  
Our proof splits into several subcases.  Recall that $B_n$ has no odd
plexes, by \lref{l:Delta}.  In each of several cases we will show that
it is possible to change a small number of the entries of $B_n$ so
that a triplex is created and yet 
there are still no transverals.
%the existence of a $1$-plex remains impossible.

%\subsection{Some small values of $n$}\label{subsub}

%Here we deal with exceptions $n\in\{6, 10,14, 18, 22, 26, 34, 38, 46, 50, 62\}$. Computer???
%I do have a solution for $10$ by hand which I could add to the paper. 
%I can maybe get $62$ to work if this is too hard by computer, need to be careful about zig-zags intersecting $J$. But the other cases I can't do by varying above methods. (Even if the trades work eg with just two trades the delta argument fails). 

First we construct transversal-free latin squares that have a triplex,
for certain small orders that are missed by the general constructions
that we will subsequently give. The examples were found by asking a
computer to complete a partial triplex in a specific latin square
$L_n$.  The construction of $L_n$ is as follows. Let $n=4m+2$. For
$i,j\in\{0,1,\dots,n-1\}$, define
\[
L_n[i,j]\text{ mod }n\equiv\begin{cases}
i+j+1&\text{if }n-1-m\le i\le n-2\text{ and }j\in\{m,3m+1\},\\
i+j-1&\text{if }n-m\le i<n\text{ and }j\in\{m+1,3m+2\},\\
i+j+m&\text{if }i=n-1-m\text{ and }j\in\{0,m+1,2m+1,3m+2\},\\
i+j-m&\text{if }i=n-1\text{ and }j\in\{0,m,2m+1,3m+1\},\\
i+j&\text{otherwise}.
\end{cases}
\]

\begin{theorem}\label{t:small}
For $n\in\{10,14,18,22,26,34,38,46,50,62\}$, there is a triplex but no trans\-versal in $L_n$.
\end{theorem}

\begin{proof}
Let $h=n/2$ and $m=(n-2)/4$. Suppose that $T$ is a transversal of $L_n$.
By \lref{l:Delta} the sum, $S$, of the $\Delta_1$ function over $T$ 
must be $h\mod n$.
However $T$ can have at most one entry in row $n-1-m$ and at most one entry
in each of columns $m,3m+1$. It follows that $S\le m+2<2m+1=h$. Similarly, 
$T$ can have at most one entry in row $n-1$ and at most one entry
in each of columns $m+1,3m+2$, so $S>-h$. It follows that $S\not\equiv h\mod n$,
so $L_n$ has no transversal.

Next, we specify a triplex $P$ in $L_n$ as follows. We start by choosing
the cells in columns $i,i+h-1,i+h$ of row $i$ for $0\le i<h$.
Next we choose the cells in columns $3(i-h),3(i-h)+1,3(i-h)+2$ of row $i$ for 
$h\le i<h+\lfloor n/6\rfloor$. For the rows with index 
$h+\lfloor n/6\rfloor$ to $n-1$ (in that order)
we list the column indices of the cells to choose in $E_n$, where $E_n$
is as follows:
\begin{align*}
E_{10}&=[[6,7,9],[0,5,8],[1,2,3],[3,4,9]],\\[1.6ex]
E_{14}&=[[6,10,11],[0,4,13],[7,8,13],[3,5,9],[1,2,12]],\\[1.6ex]
E_{18}&=[[1,13,14],[0,5,10],[11,15,16],[6,9,12],[4,7,17],[2,3,17]],\\[1.6ex]
E_{22}&=[[17,18,21],[14,15,16],[12,20,21],[0,8,9],[1,7,9],[4,6,10],[2,3,13],[5,11,19]],\\[1.6ex]
E_{26}&=[[19,20,25],[13,16,21],[14,18,22],[9,15,23],[2,11,17],[8,24,25],\\[0.1ex]&\qquad
[1,10,12],[3,4,5],[0,6,7]],\\[1.6ex]
E_{34}&=[[2,25,26],[27,28,32],[23,29,33],[20,21,24],[13,17,19],[12,14,15],\\[0.1ex]&\qquad
[10,15,16],[3,6,11],[1,7,33],[18,30,31],[4,5,22],[0,8,9]],\\[1.6ex]
E_{38}&=[[18,19,20],[21,23,25],[26,28,29],[0,10,27],[30,31,32],[33,34,35],\\[0.1ex]&\qquad
[2,36,37],[3,22,37],[1,6,24],[4,7,14],[8,9,11],[15,16,17],[5,12,13]],\\[1.6ex]
E_{46}&=[[1,34,35],[36,37,38],[31,33,39],[29,40,42],[26,27,45],[19,20,28],\\[0.1ex]&\qquad
[21,23,25],[16,18,21],[13,14,43],[2,3,4],[17,44,45],[5,6,8],[9,30,41],\\[0.1ex]&\qquad
[7,10,24],[15,22,32],[0,11,12]],\\[1.6ex]
E_{50}&=[[2,37,38],[39,40,49],[32,36,44],[33,41,42],[30,31,35],[25,27,43],\\[0.1ex]&\qquad
[23,24,47],[20,21,29],[16,17,48],[3,22,49],[1,4,5],[7,15,46],[8,10,14],\\[0.1ex]&\qquad
[6,11,19],[9,26,28],[18,34,45],[0,12,13]],\\[1.6ex]
E_{62}&=[[30,31,32],[1,46,47],[48,49,50],[39,40,42],[51,52,54],[41,45,55],[43,53,56],\\[0.1ex]&\qquad
[35,37,59],[33,38,60],[3,5,61],[6,34,61],[7,27,28],[24,25,26],[4,9,21],
\\[0.1ex]&\qquad
[10,12,58],[17,18,19],[13,14,20],[8,11,23],[22,36,57],[2,29,44],[0,15,16]].
\end{align*}
It is immediate from our construction that $P$ contains exactly $3$ entries
in each row. It is routine to check that $P$ also has exactly $3$ entries
in each column, and $3$ copies of each symbol in $N_n$.
\end{proof}

%\subsection{$n$ is divisible by $4$} 

Next we consider the case when $n$ is divisible by $4$.

\bigbreak

We start by identifying a subset of $B_n$ which we denote by $J$. 
%Since $J\subset B_n$ it suffices to identify only the cells of $J$ which are filled. 
We let $J=J_0\cup J_1\cup J_2\cup J_3\cup J_4$ where:
\begin{align*}
J_0 & = \big\{(0;0),(0;1),(0;2),(0;3),(0;4)\big\} \\
J_1 & = \big\{(i;3i+2),(i;3i+3),(i;3i+4): 1\leq i\leq n/4-1\big\} \\
J_2 & = \big\{(n/4;3n/4+2)\big\} \\
J_3 & = \big\{(i;3i),(i;3i+1),(i;3i+2): n/4+1\leq i\leq n/2-1 \big\} \\
J_4 & = \big\{(i;i-n/2+1),(i;i),(i;i+1): n/2\leq i\leq n-1\big\}. 
\end{align*}

We exhibit $J$ when $n=12$:
\[\begin{array}{|C|C|C|C|C|C|C|C|C|C|C|C|}
\hline
0 & 1 & 2 & 3 & 4 & & & & & & & \\
\hline
& & & & & 6 & 7 & 8 & & & & \\
\hline
& & & & &  &  &  & 10 & 11 & 0 & \\
\hline
& & & & &  & &  & & & & 2 \\
\hline
4 & 5 & 6 & & &  & &  & & & &  \\
\hline
 &  &  & 8 & 9 & 10 & &  & & & &  \\
\hline
 & 7 &  &  &  & & 0 & 1 & & & &  \\
\hline
 &  & 9 &  &  & &  & 2 & 3 & & &  \\
\hline
 &  &  & 11 &  & &  &  & 4 & 5 & &  \\
\hline
 &  &  &  & 1 & &  &  &  & 6 & 7 &  \\
\hline
 &  &  &  &  & 3 &  &  &  &  & 8 & 9 \\
\hline
11 &  &  &  &  & & 5 &  &  &  &  & 10 \\
\hline
\end{array}
\]

\begin{lemma}\label{l:almost}
Each column of $B_n$ contains precisely $3$ elements of $J$. Each symbol
in $N_n$ appears in precisely $3$ elements of $J$. Each row of $J$
contains precisely $3$ elements of $J$, except for the first row which
contains $5$ elements and row $n/4$ which contains $1$ element.
\end{lemma}

\begin{proof}
First observe that $|J|=|J_0|+|J_1|+|J_2|+|J_3|+|J_4|= 
5+3(n/4-1)+1+3(n/4-1)+3(n/2)=3n$. 
The statements about rows are easy to check. Next consider the
columns.  Columns $0$ through to $4$ appear once each in $J_0$;
columns $5$ through to $3n/4+1$ appear once each in $J_1$; column
$3n/4+2$ appears once in $J_2$; columns $3n/4+3$ through to $n-1$ and
$0$ through to $n/2-1$ appear in $J_3$.  Including 
elements of $J_4$ on the main diagonal, we have each column appearing exactly
twice. The set $\{i+1,i-n/2+1: n/2\le i<n\}=N_n$ so each column
of $J$ has exactly $3$ filled cells.
  
Next, consider symbols. Observe that $J_4$ contains each odd symbol
exactly twice and each even symbol exactly once.  Meanwhile, $J_1$
contains exactly one copy of the symbols from $6$ to $n-1$ (and $0$)
except for those congruent to $1$ modulo $4$.  Also, $J_3$ contains
exactly one copy of the symbols from $4$ to $n-2$ except for those
congruent to $3$ modulo $4$. Since $1,3\in J_0$, each odd symbol
occurs thrice in $J$.  Finally, since $0,2,4\in J_0$ and $2\in J_2$
each even symbol also occurs thrice in $J$.
\end{proof}

We next describe how to change $B_n$ to obtain a latin square with a
triplex but no transversal.  For each $i\in \{0,1\}$, let $T_i\subset B_n$ 
be the latin trade of cardinality $8$ consisting of all symbols in rows
$0$ and $n/4$ which are congruent to $i$ modulo $n/4$. The (unique) disjoint
mates $T_i'$ are obtained by swapping the two symbols in each column of
$T$.  In each case we will show that one of the following latin
squares will have the desired properties:
\begin{align*}
L_1 & := (B_n\setminus T_0)\cup T_0', \\
L_2 & := (B_n\setminus T_1)\cup T_1', \\
L_3 & := (B_n\setminus (T_0\cup T_1)) \cup T_0' \cup T_1'.  
\end{align*} 

\begin{lemma}
Let $n$ be divisible by $4$ and let $n\geq 8$. 
For each $i\in \{1,2,3\}$, the latin square $L_i$ does not 
contain a transversal. 
If $n=8$, $L_2$ contains a triplex. 
If $n\in \{12,16\}$, $L_1$ contains a triplex. 
If $n\geq 20$, $L_3$ contains a triplex.  
\end{lemma}

\begin{proof}
In each case, only rows $0$ and $n/4$ contain entries distinct from 
those in $B_n$. 
Thus $\Delta_1(r,c,s)=0$ if $r\not\in \{0,n/4\}$; $\Delta_1(r,c,s)\in\{0,n/4\}$ if 
$r=0$ and $\Delta_1(r,c,s)\in\{0,-n/4\}$ if $r=n/4$. Since a transversal contains 
exactly one entry in each row, if $K$ is a transversal,
$$\left\lvert \sum_{\smash{(}r\smash{,c,s)\in K}} \Delta_1(r,c,s)\right\lvert\leq n/4,$$
contradicting \lref{l:Delta}. 
The remaining claims follow from \lref{l:almost}, by observing that in each
case we have shifted precisely $2$ entries from row $0$ of $J$ to row
$n/4$ (and have made no other changes to $J$).
\end{proof}

\begin{corollary}
Let $n$ be divisible by $4$ and let $n\geq 8$. There exists a latin square of
order $n$ which contains a triplex but no transversal. 
\end{corollary}

It remains to consider the case when $n\equiv2\mod4$. This splits 
into subcases according to the value of $n\mod 3$.

The $n\equiv6\mod12$ case is easiest.  Let $m=n/6$. The $m=1$ case is
well known (for an explicit example, see \cite{origplex}).  The case
$m=3$ is done in \tref{t:small}.  For odd $m\geq 5$, we may apply
\tref{t:kk2}.

Next we consider the case when $n\equiv 10\mod12$. 
Let $n=12m-2$.  By \tref{t:small} we may assume that $m\geq 5$. 

We start by identifying a subset of $B_n$ which we denote by $J$. 
%Since $J\subset B_n$ it suffices to identify only the cells of $J$ which are filled. 
We let $J=J_0\cup J_1\cup J_2\cup J_3$ where:
\begin{align*}
J_0 & = \big\{(i;3i-2),(i;3i-1),(i;3i):  1\leq i\leq 2m-1 \big\} \\
J_1 & = \big\{(2m-1;6m+1),(2m;6m-2),(2m;6m-1)\big\} \\
J_2 & = \big\{(i;3i+2),(i;3i+3),(i;3i+4): 2m\leq i\leq n/2-1 \big\} \\
J_3 & = \big\{(i;i-n/2),(i;i),(i;i+1): n/2\leq i\leq n-1\big\}. 
\end{align*}

Observe the following lemma. We omit the proof, which is elementary
and similar to that of \lref{l:almost}.

\begin{lemma}%\label{almost2}
Each column of $B_n$ contains precisely $3$ elements of $J$. Each symbol
in $N_n$ appears in precisely $3$ elements of $J$. Each row of $B_n$
contains precisely $3$ elements of $J$, except for the first row which
contains no elements, row $2m-1$ which contains $4$ elements and row $2m$ 
which contains $5$ elements of $J$.
\end{lemma}

As in previous cases we wish to use latin trades in $B_n$ to
create a triplex without introducing a transversal.  To
this end we describe the following latin trades $T_0$ and $T_1$ within
the first $2m$ (respectively, $2m+1$) rows of $B_n$.
%; as previously it suffices to specify the cells only.
\begin{align*}
T_0 & = \big\{(i;2jm),(i;2jm+1): 1\leq j\leq 4,\,\; 0\leq i\leq 2m-1 \big\} \\
& \qquad \cup \big\{(0;1),(2m-1;1),(0;10m),(2m-1;10m)\big\}. \\
T_1 & = \big\{(0;2jm-2),(2m;2jm-2): 1\leq j\leq 4\big\}\\
& \qquad \cup\big\{(0;10m-2),(2m;12m-4)\big\} \\
& \qquad \cup \big\{(2i;12m-4-2i),(2i+2;12m-4-2i) : 0\leq i\leq m-1\big\}.
\end{align*} 

To verify that $T_0$ and $T_1$ each give latin trades in $B_n$ we exhibit 
their respective (unique) disjoint mates $T_0'$ and $T_1'$. 
\begin{align*}
T_0' & = \big\{(i,2jm,i+2jm+1),(i,2jm+1,i+2jm): 1\leq j\leq 4,\; 0< i< 2m-1 \big\} \\
& \qquad \ \cup \big\{(0,2jm,2jm+1),(0,2jm+1,2(j+1)m),(2m-1,2jm,2jm),\\
& \qquad \qquad \ (2m-1,2jm+1,2(j+1)m-1): 1\leq j\leq 4\big\} \\
& \qquad \ \cup \big\{(0,1,2m),(2m-1,1,1),(0,10m,1),(2m-1,10m,10m)\big\}. \\
T_1' & = \big\{(0,2jm-2,2(j+1)m-2),(2m,2jm-2,2jm-2): 1\leq j\leq 4\big\} \\
& \qquad \cup \big\{(0,10m-2,12m-4),(0,12m-4,2m-2)\big\} \\
& \qquad \cup \big\{(2m,10m-2,10m-2),(2m,12m-4,0)\big\} \\
& \qquad \cup \big\{(2i,12m-4-2i,0),(2i,12m-2-2i,12m-4) : 1\leq i\leq m-1\big\}. 
\end{align*} 

Lastly we define $T_2=\big\{(r,c+5,s+5): (r,c,s)\in T_1\big\}$ which
is clearly a latin trade in $B_n$ with disjoint mate
$T_2'=\big\{(r,c+5,s+5): (r,c,s)\in T_1'\big\}$.

Since $m\geq 5$, the latin trades $T_0$, $T_1$ and $T_2$ are pairwise disjoint. 
Moreover, $(T_1\cup T_2)\cap J=\big\{(2m;6m-2),(2m;6m+3)\big\}$. 
Next, $J$ intersects $T_0$ at $\big\{(2m-1;6m+1)\big\}$ and
\begin{equation}\label{e:out10mod12}
\big\{(\lceil x/3\rceil;x):x\in\{2m,2m+1,4m,4m+1\}\big\}.
\end{equation}
% once in each of the columns $2m$, $2m+1$, $4m$ and $4m+1$. 
It follows that $L'=\big(B_n\setminus (T_0\cup T_1\cup T_2)\big)\cup
(T_0'\cup T_1'\cup T_2')$ contains a triplex. To see this, adjust $J$
by replacing $(2m,6m-2,8m-2)$, $(2m,6m+3,8m+3)$ and $(2m-1,6m+1,8m)$
with $(0,6m-2,8m-2)$, $(0,6m+3,8m+3)$ and $(0,6m+1,8m)$, respectively.
Finally, 
%transpose the elements between columns $2m$ and $2m+1$ of $J$ and also
%transpose the elements between columns $4m$ and $4m+1$ of $J$.
replace \eref{e:out10mod12} with the triples of $L'$ associated with the
following cells:
\[
\big\{
(\lceil 2m/3\rceil,2m+1),(\lceil(2m+1)/3\rceil,2m),
(\lceil 4m/3\rceil,4m+1),(\lceil(4m+1)/3\rceil,4m)
\big\}
\]
The resultant structure is a triplex in $L'$.

Suppose, for the sake of contradiction, $L'$ has a transversal
$K$. Recall in the following that $K$ intersects each row, column and
symbol exactly once.
If $(r,c,s)\in T_0'\cup T_1' \cup T_2'$, $1\leq \Delta_1(0,c,s)\leq 2m$,
$\Delta_1(2m-1,c,s)\in\big\{-1,-(2m-1)\big\}$ and 
$\Delta_1(2m,c,s)\in\{-2m,-2m+2\}$.
%% WAS $-2m\leq \Delta_1(2m,c,s)\leq -1$.
Summing over $(r,c,s)\in K$ with $r\in \{0,2m-1,2m\}$:
 $$\left\lvert \sum\Delta_1(r,c,s)\right\lvert\leq 4m-1.$$ 

Otherwise the only non-zero values for $\Delta_1$ occur strictly between rows $0$ and $2m-1$ of $T_0'\cup T_1'\cup T_2'$,
so we consider only when $0<r<2m-1$. In this case 
 $\Delta_1(r,c,s)=1$ if $c\in C=\{2m,4m,6m,8m\}$ and
$\Delta_1(r,c,s)=-1$ if $c\in C'=\{2m+1,4m+1,6m+1,8m+1\}$.
Summing over $(r,c,s)\in K$ with $r\not\in\{0,2m-1,2m\}$ and  $c\in C\cup C'$:
 $$\left\lvert\sum\Delta_1(r,c,s)\right\lvert\leq 4.$$ 

Similarly, the most the $\Delta_1$ function can accrue from $T_1'\cup
T_2'$ in rows strictly between $0$ and $2m-1$ is $4$ (in absolute
terms). Thus:
$$\left\lvert\sum_{\smash{(}r\smash{,c,s)\in K}} \Delta_1(r,c,s)\right\lvert\leq 4m+7<n/2,$$ 
a contradiction. 

%\subsection{$n-2$ is divisible by $12$}

\bigskip

Finally, we consider the case when $n\equiv 2\mod12$. 
Let $n=12m+2$.  By \tref{t:small} we may assume that $m\geq 6$. 

We start by identifying a subset of $B_n$ which we denote by $J$.
%Since $J\subset B_n$ it suffices to identify only the cells of $J$
%which are filled.  
We let $J=J_0\cup J_1\cup J_2\cup J_3$ where:
\begin{align*}
J_0 & = \big\{(i;3i-2),(i;3i-1),(i;3i):  1\leq i\leq 2m \big\} \\
J_1 & = \big\{(2m;6m+3),(2m;6m+4),(2m+1;6m+1)\big\} \\
J_2 & = \big\{(i;3i+2),(i;3i+3),(i;3i+4): 2m+1\leq i\leq n/2-1 \big\} \\
J_3 & = \big\{(i;i-n/2),(i;i),(i;i+1): n/2\leq i\leq n-1\big\}. 
\end{align*}

Observe the following lemma. We omit the proof, which is elementary
and similar to that of \lref{l:almost}.

\begin{lemma}\label{l:almost3}
Each column of $B_n$ contains precisely $3$ elements of $J$. Each symbol
in $N_n$ appears in precisely $3$ elements of $J$. Each row of $B_n$
contains precisely $3$ elements of $J$, except for the first row which
contains no elements, row $2m$ which contains $5$ elements
and row $2m+1$ which contains $4$ elements of $J$.
\end{lemma}

As in previous cases we wish to use latin trades in $B_n$ 
to introduce a triplex but not a transversal.  To
this end we describe the following latin trades $T_0$ and $T_1$ within
the first $2m$ (respectively, $2m+1$) rows of $B_n$.
%; as previously it suffices to specify the cells only.
\begin{align*}
T_0 & =  
\big\{(0;2jm-2),(2m;2jm-2): 3\leq j\leq 6 \big\} \\
& \qquad \cup \big\{(i;2m-4),(i;2m-3),(i;4m-3),(i;4m-2): 0\leq i\leq 2m\big\}, \\
T_1 & = \big\{(0;(2m+1)j-2),(2m+1;(2m+1)j-2): 0\leq j\leq 4\big\}\\
& \qquad \cup \big\{(0;10m+1),(1;10m),(1;10m+1),(2;10m),(2;10m+1)\big\} \\
& \qquad \cup \big\{(2m+1;10m+1),(1;12m-1),(1;12m),(2;12m-1),(2;12m)\big\} \\
& \qquad \cup \big\{(2i+3;10m-2i-2),(2i+3;10m-2i),\\
& \qquad \qquad (2i+3;12m-2i-3),(2i+3;12m-2i-1): 0\leq i\leq m-2 \big\}.
\end{align*} 

To verify that $T_0$ and $T_1$ each give latin trades in $B_n$ we
exhibit their respective (unique) disjoint mates $T_0'$ and $T_1'$.
\begin{align*}
T_0' & =  
\big\{(0,2jm-2,2(j+1)m-2),(2m,2jm-2,2jm-2): 3\leq j\leq 6 \big\} \\
& \qquad \cup \big\{(0,2m-4,2m-3),(0,2m-3,4m-3),(2m,2m-4,2m-4)\big\} \\
& \qquad \cup \big\{(2m,2m-3,4m-4),(0,4m-3,4m-2),(0,4m-2,6m-2)\big\} \\
& \qquad \cup\big\{(2m,4m-3,4m-3),(2m,4m-2,6m-3)\big\} \\
% WAS & \qquad \cup\big\{(2m,2m-3,2m-3),(2m,2m-2,6m-3)\big\} \\
& \qquad \cup \big\{(i,2m-4,2m+i-3),(i,2m-3,2m+i-4),(i,4m-3,4m+i-2), \\
& \qquad \qquad (i,4m-2,4m+i-3): 0< i< 2m\big\}. \\ 
T_1' & =  \big\{(0,(2m+1)j-2,(2m+1)(j+1)-2), \\
& \qquad \qquad (2m+1,(2m+1)j-2,(2m+1)j-2): 1\leq j\leq 3\big\}\\
& \qquad \cup\big\{(0,12m,2m-1),(2m+1,12m,0),(0,8m+2,10m+1)\big\} \\
& \qquad \cup\big\{(2m+1,8m+2,8m+2),(0,10m+1,12m),(1,10m,10m+2)\big\} \\
& \qquad \cup \big\{(1,10m+1,10m+1),(2,10m,10m+3),(2,10m+1,10m+2)\big\} \\
& \qquad \cup \big\{(2m+1,10m+1,10m+3),(1,12m-1,12m+1),(1,12m,12m)\big\} \\
& \qquad \cup\big\{(2,12m-1,0),(2,12m,12m+1)\big\} \\
& \qquad \cup \big\{(2i+3,10m-2i-2,10m+3),(2i+3,10m-2i,10m+1), \\
& \qquad \qquad (2i+3,12m-2i-3,0),(2i+3,12m-2i-1,12m): 0\leq i\leq m-2 \big\}.
\end{align*}

We also define $T_2=\big\{(r,c+6,s+6): (r,c,s)\in T_0\big\}$ which is clearly a latin trade in $B_n$ with disjoint mate $T_2'=\big\{(r,c+6,s+6): (r,c,s)\in T_0'\big\}$.  

Let $m\geq 6$. Observe that the latin trades $T_0$, $T_1$ and $T_2$ are pairwise disjoint. 
Moreover, $T_0\cup T_2$ intersects $J$ at $(2m;6m-2),(2m;6m+4)$ and 
%once in each of the columns $2m-4,2m-3,2m+2,2m+3,4m-3,4m-2,4m+3$ and $4m+4$.  
\begin{equation}\label{e:out2mod12}
\big\{
(\lceil x/3\rceil;x):x\in\{2m-4,2m-3,2m+2,2m+3,4m-3,4m-2,4m+3,4m+4\}
\big\}
\end{equation}
Also, $J$ intersects $T_1$ at $\big\{(2m+1;6m+1)\big\}$. 

It follows that $L'=\big(B_n\setminus (T_0\cup T_1\cup T_2)\big)\cup 
(T_0'\cup T_1'\cup T_2')$ contains a triplex. To see this, adjust $J$ by 
replacing $(2m,6m-2,8m-2)$, $(2m,6m+4,8m+4)$ and 
$(2m+1,6m+1,8m+2)$ with $(0,6m-2,8m-2)$, $(0,6m+4,8m+4)$ and 
$(0,6m+1,8m+2)$, respectively. Finally, 
%transpose the elements in $J$ between each of the pairs of columns 
%$2m-4$ and $2m-3$, $2m+2$ and $2m+3$, $4m-3$ and $4m-2$, $4m+3$ and $4m+4$.
replace \eref{e:out2mod12} with the triples of $L'$ associated with the
following cells:
\begin{align*}
\big\{
&(\lceil(2m-4)/3\rceil,2m-3),(\lceil(2m-3)/3\rceil,2m-4),\\
&(\lceil(2m+2)/3\rceil,2m+3),(\lceil(2m+3)/3\rceil,2m+2),\\
&(\lceil(4m-3)/3\rceil,4m-2),(\lceil(4m-2)/3\rceil,4m-3),\\
&(\lceil(4m+3)/3\rceil,4m+4),(\lceil(4m+4)/3\rceil,4m+3)
\big\}.
\end{align*}
The resultant structure is a triplex in $L'$. 

Suppose, for the sake of contradiction, $L'$ has a transversal $K$. 
If $(r,c,s)\in T_0'\cup T_1' \cup T_2'$, $1\leq \Delta_1(0,c,s)\leq 2m+1$,
$\Delta_1(2m,c,s)\in\{-1,-2m\}$ and 
$\Delta_1(2m+1,c,s)\in\{-2m-1,-2m+1\}$.
%% WAS $-(2m+1)\leq \Delta_1(2m+1,c,s)\leq -2$.
Summing over $(r,c,s)\in K$ with $r\in \{0,2m,2m+1\}$:
 $$\left\lvert\sum \Delta_1(r,c,s)\right\lvert\leq 4m+1.$$

Otherwise the only non-zero values for $\Delta_1$ occur strictly between rows $0$ and $2m$ of $T_0'\cup T_1'\cup T_2'$,
so we consider only when $0<r<2m$. In this case 
 $\Delta_1(r,c,s)=1$ if $c\in C=\{2m-4,2m+2,4m-3,4m+3\}$ and
$\Delta_1(r,c,s)=-1$ if $c\in C'=\{2m-3,2m+3,4m-2,4m+4\}$.
Summing over $(r,c,s)\in K$ with $r\not\in\{0,2m,2m+1\}$ and  $c\in C\cup C'$:
 $$\left\lvert\sum \Delta_1(r,c,s)\right\lvert\leq 4.$$ 
Similarly, the most the $\Delta_1$ function can accrue from $T_1'$ in rows 
strictly between $0$ and $2m$ is $6$ (in absolute terms). Thus:
$$
\left\lvert\sum_{\smash{(}r\smash{,c,s)\in K}} \Delta_1(r,c,s)\right\lvert\leq 4m+11<n/2,
$$ 
a contradiction.

\section{Conclusion}

It has been known since the 19th century that there are no
transversals in step-type Latin squares of even order composed of odd
ordered subsquares. This family includes the Cayley table of the
cyclic group of any even order.  We showed in \sref{s:numnotrans} that
the absence of transversals in these squares is a surprisingly robust
property.  Specifically, the entries in up to $\sqrt{n}$ consecutive
rows may be rearranged in any way and there will still be no
transversal. A consequence is that there are at least
$n^{n^{3/2}(1/2-o(1))}$ species of transversal-free latin squares of
each even order $n$.

Our other main result was to prove \cjref{cj:triplexnotrans}, that 
for all even $n>4$ there is a latin square of order $n$ that contains
a triplex but no transversal. It would be interesting to know how far
this result generalises.  We propose:

\begin{conj}
For each odd $k$ there exists $N$
such that for all even $n\ge N$ there exists a latin square of order
$n$ that contains a $k$-plex but no $k'$-plex for odd $k'<k$.
\end{conj}

%We showed this property holds for $k=3$.

In \cite{EW08}, examples were constructed where the smallest odd
$k$-plexes have $k=n/4+O(1)$, where $n$ is the order of the
latin square.  This raises the interesting question of whether
$n/4+O(1)$ is as large as possible in a result of this type. 
%there exists any latin square in which there are odd plexes but the 
%smallest odd $k$-plex has $k>n/4$.

Another unsolved question from \cite{EW08} is whether there exists
any latin square that has an $a$-plex and a $c$-plex but no
$b$-plex, for odd integers $a<b<c$.

\subsection*{Acknowledgement}
We thank Michael Brand for interesting discussions on our results
in \sref{s:numnotrans}, and Judith Egan for some useful feedback
on a draft of the paper.

\bigbreak\bigbreak

\end{document}